\documentclass{birkart}

\usepackage{amsmath,amssymb,amsthm,amscd}

   \newtheorem{Thm}{Theorem}[subsection]

   \newtheorem{Def}[Thm]{Definition}

   \makeatletter
   \@addtoreset{equation}{section}
   \makeatother
   \renewcommand{\theequation}{\arabic{section}.\arabic{equation}}
\newcommand{\R}{\mathbb{R}}
\newcommand{\N}{\mathbb{N}}
\newcommand{\Z}{\mathbb{Z}}
\newcommand{\C}{\mathbb{C}}

\newcommand{\V}{\mathbb{V}}
\newcommand{\M}{\mathbb{M}}

\newcommand{\dbar}{{d\hspace{-0,05cm}\bar{}\hspace{0,05cm}}}
\newcommand{\Op}{\textup{Op}}
\newcommand{\op}{\textup{op}}
\newcommand{\cl}{\textup{cl}}

\newcommand{\g}{\textbf{g}}

\newcommand{\ci}{C^\infty}
\newcommand{\Hsum}[2]{\begin{matrix} #1 \\ \oplus \\
     #2 \end{matrix}}

\begin{document}
%
\title[Parameter-dependent Edge Operators]
{Parameter-dependent Edge Operators}
\author{C.-I. Martin}

\address{%
UCLA Mathematics Department
Los Angeles, CA 90095-1555\\
USA}

\email{cmartin@math.ucla.edu}

\author{B.-W. Schulze}
\address{%
Am Neuen Palais 10\\
14469 Potsdam\\
Germany}
\email{schulze@math.uni-potsdam.de}



\begin{abstract}
We study parameter-dependent operators on a manifold with edge and construct new classes of elliptic elements in the corner calculus on an infinite cone with a singular base.

\end{abstract}

\maketitle

\section{Introduction}
The ellipticity of a (pseudo-)differential operator $A$ on a manifold $M$  with singularities (more precisely, a stratified space with strata of conical, edge, boundary, corner, $\dots$, type) is a condition on the components of the principal symbolic hierarchy $\sigma(A)=(\sigma_j(A))_{0\leq j\leq k}$ belonging to $A.$ Here $k\in\N$ is the order of the singularity where $k=0$ corresponds to smoothness, $k=1$ to the conical, edge, or boundary case, etc., cf. \cite{Schu27}, \cite{Schu57}. Let $s(M)=(s_j(M))_{0\leq j\leq k}$ denote the sequence of strata of $M.$ The component $\sigma_0(A)$ has the meaning of the standard homogeneous principal symbol of $A$; it is a function in $\ci(T^*s_0(M)\setminus0)$ where $s_0(M)$ is the main stratum of $M$ (the one of maximal dimension). Every $\sigma_j(A)$ for $j\geq 1$ is a family of operators between suitable distribution spaces on a manifold with singularities of order $j-1,$ determined by the corresponding $\sigma_j(M).$ By ellipticity of $A$ we understand the pointwise invertibility of $\sigma_j(A)$ (and of a certain reduced symbol $\tilde{\sigma}_j(A)$) for all $j$. We also talk about ($\sigma_j)_{0\leq j\leq i}$-ellipticity when the bijectivities refer to $0\leq j\leq i.$ To assess the solvability of a $\sigma_0$-elliptic equation it is essential to observe also the other symbolic components (e.g., on a manifold with boundary it is well-known that the boundary symbol $\sigma_1(A)$ is responsible for elliptic boundary conditions). It can be very difficult to decide whether a $\sigma_0$- or ($\sigma_j)_{0\leq j\leq i}$- elliptic operator is also elliptic with respect to the remaining $(\sigma_j)_{i+1\leq j\leq k}$ (which is, of course, not the case in general). The articles \cite{Dine4} and \cite{Dine3} answer such a question in the case when edge or (second order) corner singularities are embedded in a smooth manifold.\\The present paper gives explicit constructions of new classes of elliptic operators on an infinite cone $M:=B^\Delta=(\overline{\R}_+\times B)/(\{0\}\times B)$ with a base $B$ which is a compact manifold with edges. The space $B^\Delta$ is singular of order $k=2.$ We mainly focus on $\sigma_2$ which is associated with the corner point, represented by $\{0\}\times B$ in the quotient space. Given any prescribed corner weight we show that a $(\sigma_0,\sigma_1)$-elliptic operator admits a $\sigma$-elliptic representative in the class modulo so-called smoothing Mellin operators in the algebra of corner pseudo-differential operators. In addition if an operator in our calculus is only $\sigma_0$-elliptic we obtain $\sigma_1$-ellipticity locally in any neighbourhood of an edge point on $M,$ using a result of \cite{Mali1} for the case of a cone with smooth base.\\Assume for the moment that $M$ is a manifold with conical singularity $\{c\},$ locally near $\{c\}$ modelled on $X^\Delta$ for a closed $\ci$ manifold $X.$ Then by cone calculus on $M$ we understand a substructure of $L^m_{\textup{cl}}(M\setminus\{c\})$ consisting of operators $A$ that are $\textup{mod}\,L^{-\infty}(M\setminus\{c\})$ in the splitting of variables $(r,x)\in\R_+\times X=:X^\wedge$ of the form $r^{-m}\mbox{Op}_r(p),\, \mbox{Op}_r(p):= F^{-1}p(r,\rho)F,$ with $F=F_{r\rightarrow\rho}$ being the one-dimensional Fourier transform, and $p(r,\rho)=\tilde{p}(r,r\rho),\, \tilde{p}(r,\tilde{\rho})\in L^m_{\textup{cl}}(X;\R_{\tilde{\rho}})$ (cf. the notation in subsection 1.1 below). The smoothing remainders will be specified in the context of a suitable Mellin quantisation in $r,$ and other contributions. 
Examples are differential operators of Fuchs type  
\begin {equation} \label{ecC}
A=r^{-m}\sum_{j=0}^ma_j(r)( -r\partial_r)^j, 
\end {equation}
$a_j\in \ci (\overline{\R}_+, \mbox{Diff}^{m -j}(X))$; here $\mbox{Diff}^{\nu }(X)$ is the space of all differential operators on $X$ of order $\nu$ with smooth coefficients. Recall that when $g _X$ is a Riemannian metric on $X$ the Laplace-Beltrami operator belonging to the cone metric $dr^2+r^2g _X$ on $X^\wedge \ni (r,x)$ is of the form $\eqref{ecC}$ (for $m =2).$ Similarly, if we consider a (stretched) wedge $X^\wedge \times \Omega \ni (r,x,y),\,\Omega\subseteq\R^q$ open, and a wedge metric $dr^2+r^2g_X+dy^2,$ the associated Laplace-Beltrami operator is of the form   
\begin{equation} \label{ecE}
A=r^{-m}\sum_{j+|\alpha |\leq m}a_{j\alpha }(r,y)( -r\partial_r)^j(rD_y)^\alpha 
\end{equation}
(for $m =2$) with coefficients $a_{j\alpha }\in \ci(\overline{\R}_+\times \Omega , \mbox{Diff}^{m -(j+|\alpha |)}(X))$. Operators  
$\eqref{ecE}$ are edge-degenerate, i.e., of the form
\begin{equation} \label{ecF}r^{-m} \mbox{Op}_{r,y}(p), \end{equation}
for $ \mbox{Op}_{r,y}(p)u(r,y)= F^{-1}_{(\rho,\eta)\rightarrow (r,y)} p(r,y,\rho ,\eta )(F_{(r,y)\rightarrow (\rho,\eta)}u))(r,y),\,p(r,y,\rho ,\eta )=\tilde{p}(r,y,r\rho ,r\eta ),$ for a $ \tilde{p}(r,y,\tilde{\rho}, \tilde{\eta }) \in \ci(\overline{\R}_+\times \Omega, L_\cl ^m (X;\R_{\tilde{\rho}, \tilde{\eta }}^{1+q }))$ and the Fourier transform $F_{(r,y)\rightarrow(\rho,\eta)}$ in the variables $(r,y)\in \R\times\R^q$. Operators $\eqref{ecF}$ are the raw material of the edge calculus on $X^\wedge\times\Omega,$ a substructure of $L^m_{\textup{cl}}(X^\wedge\times\Omega).$  More generally, if $M$ is a manifold with edge $Y,$ locally near $Y$ modelled on wedges $X^\Delta\times\Omega,$ then the edge calculus consists of operators in $L^m_{\textup{cl}}(M\setminus Y)$ that are edge-degenerate near $Y,$ modulo some specific smoothing operators. Apart from the standard homogeneous principal symbol $\sigma_0(A)$ of an operator $A$ in the edge calculus, close to $Y$ we have the parameter-dependent homogeneous principal symbol $\tilde{p}_{(m)}(r,x,y,\tilde{\rho},\xi,\tilde{\eta})$ of the above-mentioned $\tilde{p}$ (with $x$ being local coordinates on $X$ with covariables  $\xi).$ Then we have $r^m\sigma_0(A)(r,x,y,\rho,\xi,\eta)=\tilde{p}_{(m)}(r,x,y,r\rho,\xi,r\eta) =: \tilde{\sigma}_0(A)(r,x,y,\rho,\xi,\eta)$ (which is valid including $q=0).$ Moreover, in the case of  differential operators $\eqref{ecC},$ $\eqref{ecE},$ the definition of $\sigma_1(A)$ for $q>0$ is
\begin{equation} \label{sy1}
\sigma_1(A)(y,\eta):=r^{-m}\sum_{j+|\alpha |\leq m}a_{j\alpha }(0,y)( -r\partial_r)^j(r\eta)^\alpha\,:\mathcal{K}^{s,\gamma}(X^\wedge)\rightarrow\mathcal{K}^{s-m,\gamma-m}(X^\wedge),
\end{equation}
$(y,\eta)\in T^*\Omega\setminus0,$ called the (principal) edge symbol of $A,$ and for $q=0$ $$\sigma_1(A)(z)=\sum_{j=0}^ma_j(0)z^j:H^s(X)\rightarrow H^{s-m}(X),$$ $z\in\Gamma_{(n+1)/2-\gamma},$ called the (principal) conormal symbol of $A,$ for $\Gamma_{(n+1)/2-\gamma}:=\{z\in\C: \textup{Re}\,z=(n+1)/2-\gamma\},$ also referred to as the weight line of weight $\gamma $. Moreover, $H^s(X)$ are the standard Sobolev spaces on $X$, while $\mathcal{K}^{s,\gamma}(X^\wedge)$ are weighted spaces of smoothness $s$ and weight $\gamma$ on the open stretched cone $X^\wedge$, cf. subsection 1.2 below.\\
We study here a number of new properties of parameter-dependent operators on a manifold with edge, in particular, with respect of their role of amplitude functions for the calculus on an infinite cone with singular base $M,$ with a holomorphic
dependence on parameters. Those occur, in particular, as conormal symbols $h$ of operators on a corner manifold. Under ellipticity with respect to the imaginary part of the complex covariable as parameter we show that for every corner weight $\delta \in \R$ there is a smoothing Mellin symbol $f$, such that $h+f$ is bijective on the weight line of weight $\delta$. The  pointwise inverses as operators between weighted edge spaces are meromorphic Fredholm families; those occur as conormal symbols of parametrices. Our result also extends to the meromorphic case.

\section{Edge calculus with parameters}
\label{cylsec1}
\subsection{Edge-degenerate operators}
Parameter-dependent operators on smooth manifolds in different contexts have been considered by many authors, see Agranovich and Vishik \cite{Agra1}, Kondratyev \cite{Kond1}, Grubb \cite{Grub1}. Here we refer to the edge calculus of \cite{Schu32}.\\
Let us first fix some notation on pseudo-differential operators on a $\ci$ manifold $M$. By $L^m_{\textup{cl}}(M;\R^l), m\in\R,$ we denote the space of classical parameter-dependent pseudo-differential operators of order $m\in\R$ on $M,$ with parameters $\lambda\in\R^l.$ More precisely, modulo a smoothing operator family, every $A\in L^m_{\textup{cl}}(M;\R^l)$ is locally of the form $A(\lambda)= F^{-1}a(x,\xi,\lambda)F$ for a classical symbol $a(x,\xi,\lambda)$ in $(\xi,\lambda)$ of order $m$ where $F=F_{x\to\xi}$ is the Fourier transform in $x\in\R^n, n=\textup{dim}\,M.$ The space of parameter-dependent smoothing operators $L^{-\infty}(M;\R^l)$ is defined as $\mathcal{S}(\R^l,L^{-\infty}(M))$ with $L^{-\infty}(M)$ being the space of smoothing operators on $M;$ the correspondence to kernels in $\ci(M\times M),$ refers to a fixed Riemannian metric on $M.$ \\ Parameter-dependent ellipticity of A means that the homogeneous principal symbol $a_{(m)}(x,\xi,\lambda)$
(of homogeneity $m$ in $(\xi,\lambda)\ne 0)$ never vanishes. It is known that then $A$ has a parametrix $P\in L^{-m}_{\textup{cl}}(M;\R^l),\,1-PA,\,1-AP \in L^{-\infty}(M;\R^l).$ We will employ the well-known fact that when $M$ is compact for every $m$ there is a parameter-dependent elliptic $R^m\in L^m_{\textup{cl}}(M;\R^l)$ which induces isomorphisms between Sobolev spaces, namely,
\begin{equation} \label{11.red}
R^m(\lambda):H^s(M)\rightarrow H^{s-m}(M)
\end{equation} for all $s\in \R$
and $\lambda\in\R^l.$ Then $(R^m)^{-1}(\lambda )\in L^{-m}_{\textup{cl}}(M;\R^l).$\\ Order reducing families of that kind can be employed to define some useful versions of Sobolev spaces on $\R\times M,$ and $\R_+\times M,$ respectively. First, if $F=F_{t\rightarrow\tau}$ is the one-dimensional Fourier transform and $R^s(\tau)$ such a family for $l=1,$ then the completion of $\ci_0(\R\times M)$ with respect to the norm $$\big\{\int\|R^s(\tau)F_{t \rightarrow\tau}u(\tau,\cdot)\|^2_{L^2(M)}\dbar\tau\big\}^{1/2}$$ gives us the cylindrical Sobolev space $H^s(\R\times M).$ Moreover, for the Mellin transform $$Mu(z)=\int_0^\infty r^{z-1}u(r)dr$$ on $\R_+$ with the complex covariable $z$ we obtain the weighted space $\mathcal{H}^{s,\gamma}(M^\wedge)$ on the open stretched cone $M^\wedge:=\R_+\times M$ as the completion of $\ci_0(M^\wedge)$ with respect to the norm \begin{equation} \label{11.weight}
\big\{\int_{\Gamma_{(n+1)/2-\gamma}}\|R^s(\textup{Im}\,z)M_{r \rightarrow z}f(z,\cdot)\|^2_{L^2(M)}\dbar z\big\}^{1/2}
\end{equation}
where $\dbar z:= (2\pi i)^{-1}dz,$ and $\Gamma_\beta:= \{z\in\C:\textup{Re}\,z=\beta\}.$ 
Spaces of that kind have been employed in Kondratyev's work \cite{Kond1} on boundary value problems on manifolds with conical singularities, and later on by many other authors. Another, more subtle, expression is the norm 
\begin{equation} \label{11.cone}\big\{\int\|\langle r\rangle^{-s+g+n/2}\big(F^{-1}_{\rho\rightarrow r}R^s(\langle r\rangle\rho,\langle r\rangle\eta)(F_{r\rightarrow\rho}u)\big)(r,\cdot)\|^2_{L^2(M)}dr\big\}^{1/2},
\end{equation} 
$\langle r\rangle:= (1+|r|^2)^{1/2},$ on the space $\ci_0(\R\times M)$ for any fixed $\eta\in\R^q$ of sufficiently large absolute value. The order reducing family now refers to the $(1+q)$-dimensional parameter $(\tilde{\rho},\tilde{\eta}).$ The completion gives us a space that we denote by $H^{s;g}_{\textup{cone}}(\R\times M),\,s,g\in\R,$ and we set $H^s_{\textup{cone}}(\R\times M):=H^{s;0}_{\textup{cone}}(\R\times M),\,H^s_{\textup{cone}}(M^\wedge):=H^s_{\textup{cone}}(\R\times M)|_{M^\wedge}.$\\ Let us now replace $M$ by a manifold with edge $Y$. Recall that such an $M$ can be represented as a quotient space
$\M/\!\!\sim$ for the stretched manifold $\M$ which is a smooth manifold with boundary $\partial\M$ that is an $X$-bundle over $Y$ for some (here closed compact) $\ci$ manifold $X.$ The equivalence relation $\sim$ is induced by the bundle projection $\partial\M\rightarrow Y$, while $\M\setminus\partial\M=\textup{int}\,\M$ is diffeomorphic to $M\setminus Y=:\textup{int}\,M.$ Locally near $\partial\M$ we have representations of $\M$ by stretched wedge neighbourhoods
\begin{equation} \label{11.wedge}
\M\supseteq\V \cong\overline{\R}_+\times X\times \Omega,
\end{equation}
$\Omega\subseteq\R^q$ open, for $q=\textup{dim}\,Y,$ where the diffeomorphism in $\eqref{11.wedge}$ is an isomorphism in the category of manifolds with edge, here simply a diffeomorphism between the respective manifolds with boundary which induces an isomorphism $\partial\V\cong X\times\Omega$ between the trivial $X$-bundles over $\partial\V$ and $\Omega,$ respectively. From $\partial\V$ we have local splittings of variables $(r,x,y)$ with corresponding covariables $(\rho,\xi,\eta).$ In our notation a manifold $M$ with edge $Y$  is not necessarily a topological manifold with continuous charts to open sets in $\R^{\textup{dim}\,M}$ but a stratified space. We have the strata 
\begin{equation} \label{strata}
s_0(M)=\textup{int}\,M,\,\, s_1(M)=Y
\end{equation}
and a disjoint union $M=s_0(M)\bigcup s_1(M).$ The above-mentioned symbolic components $\sigma_i(A)$ of an operator $A$ in our calculus are associated with $s_i(M),\,i=0,1.$\\
The raw material for operators in the pseudo-differential calculus on $M$ are parameter-dependent families $$\tilde{p}(r,y,\tilde{\rho},\tilde{\eta},\tilde{\lambda})\in\ci(\overline{\R}_+\times \Omega, L^m_{\textup{cl}}(M;\R^{1+q+l}_{\tilde{\rho},\tilde{\eta},\tilde{\lambda}})),$$ $m\in\R$. Let $L^m_{\textup{deg}}(M;\R^l)$ denote the subspace of all $A\in L^m_{\textup{cl}}(M\setminus Y;\R^l)$ that are $\textup{mod}\,L^{-\infty}(M\setminus Y;\R^l)$ locally near $\partial\M$ of the form $$A(\lambda)=r^{-m}\textup{Op}_{r,y}(p)(\lambda)$$ for an $L^m_{\textup{cl}}(X)$-valued amplitude function
\begin{equation} \label{11.amp}
 p(r,y,\rho,\eta,\lambda):=\tilde{p}(r,y,r\rho,r\eta,r\lambda). 
\end{equation}
We set $L^{-\infty}_{\textup{deg}}(M;\R^l):=L^{-\infty}(M\setminus Y;\R^l).$ Let $\sigma_0(A)$ denote the parameter-dependent homogeneous principal symbol of $A\in L^m_{\textup{deg}}(M;\R^l),\,\sigma_0(A)\in\ci(T^*(M\setminus Y)\times\R^l\setminus 0).$ Locally near $Y$ in the variables $(r,x,y)$ there is a function $\tilde{\sigma}_0(A)(r,x,y,\tilde{\rho},\xi,\tilde{\eta},\tilde{\lambda})$ (that we also call the reduced symbol), smooth in $r$ up to $0,$ such that $$\sigma_0(A)(r,x,y,\rho,\xi,\eta, \lambda)=r^{-m}\tilde{\sigma}_0(A)(r,x,y,r\rho,\xi,r\eta,r\lambda).$$ Similarly as in the standard pseudo-differential calculus over $M\setminus Y$ every $A\in L^m_{\textup{deg}}(M;\R^l)$ can be represented as a sum $A=A_0+C$ where $A_0$ is properly supported and $C\in L^{-\infty}(M\setminus Y;\R^l).$\\Observe that for every $A_i\in L^{m_i}_{\textup{deg}}(M;\R^l),\,i=1,2,$ and $A_1$ or $A_2$ properly supported, we have $A_1A_2\in L^{m_1+m_2}_{\textup{deg}}(M;\R^l),$ and that the parameter-dependent homogeneous principal symbols are multiplicative. Also other elements of the standard calculus such as the behaviour of formal adjoints, or the invariance (here under isomorphisms in the category of manifolds with edge) can easily be checked.\\By the (parameter-dependent) edge algebra we understand a specific subalgebra of \,$\bigcup_mL^m_{\textup{deg}}(M;\R^l)$, characterised by certain conditions close to the edge $Y.$ Those contain weight and asymptotic data, and some quantisation of the local amplitude functions $\eqref{11.amp}$ that turn the resulting elements to (families of) continuous operators
\begin{equation}\label{11.cont}
A(\lambda):H^{s,\gamma}(M)\rightarrow H^{s-m,\gamma-\mu}(M)
\end{equation}
in weighted edge spaces $H
^{s,\gamma}(M)$ of smoothness $s$ and weight $\gamma,$ cf. \cite{Schu20}, or subsection 1.2 below. The weight shift $\mu$ and its relationship with $m$ will be explained below. We will interpret $\mu$ as the leading order in connection with ellipticity, while lower order terms with respect to edge symbols will be characterised by $m=\mu-j,\,j\in\N.$ Moreover, together with the pair of weights $(\gamma,\gamma-\mu)$ we will keep in mind a fixed width of a half-open weight strip $\Theta=(-(k+1),0], \,k\in\N\cup\,\{\infty\}$ on the left of the weight lines $\Gamma_{n+1/2-\gamma}$ and $\Gamma_{n+1/2-(\gamma-\mu)},$ respectively, $n=\textup{dim}\,X.$ Let ${\bf{g}}_{\gamma,\mu}:=(\gamma,\gamma-\mu,\Theta)$ denote the weight data. Our edge operators will be defined in terms of a sum
\begin{equation}\label{11.alg}
L^m(M,{\bf{g}}_{\gamma,\mu};\R^l)=L^m_{\textup{flat}}(M,{\bf{g}}_{\gamma,\mu};\R^l)+L^m_{\textup{as}}(M,{\bf{g}}_{\gamma,\mu};\R^l).
\end{equation}
The subspaces $L^m_{\textup{flat}}$ and $L^m_{\textup{as}}$ will be studied below in the subsections 1.2 and 1.3, based on holomorphic and meromorphic Mellin symbols. For the holomorphic part we need the following definition. If $E$ is a Fr\'echet space and $U\subseteq\C$ open, by $\mathcal{A}(U,E)$ we denote the space of all holomorphic functions in $U$ with values in $E,$ in the Fr\'echet topology of uniform convergence on compact subsets.
\begin{Def} \label{11.mehol}
Let $M^m_{\mathcal{O}}(X;\R^l)$ defined to be the set of all $h(z,\lambda)\in\mathcal{A}(\C,L^m_{\textup{cl}}(X;\R^l))$ such that $h(\beta+i\rho,\lambda)\in L^m_{\textup{cl}}(X;\R^{1+l}))$ for every $\beta\in\R,$ uniformly in compact $\beta$-intervals.
\end{Def}
The following Mellin quantisation result from \cite{Schu2} refers to pseudo-differential operators on $\R_+$ based on the weighted Mellin transform, namely,
\begin{equation}\label{11.mps}
\textup{op}^\delta_M(f)(\lambda)=\iint(r/r')^{-(1/2-\delta+i\rho)}f(r,r',1/2-\delta+i\rho,\lambda)u(r')dr'/r'\dbar \rho
\end{equation}
with $f(r,r',z,\lambda)\in \ci(\R_+\times\R_+,L^m_{\textup{cl}}(X;\Gamma_{1/2-\delta}\times \R^l))$ being an $L^m_{\textup{cl}}(X;\R^l)$-valued amplitude function.
\begin{Thm} \label{1.ph}
For every $\tilde{p}(r,r',y,\tilde{\rho},\tilde{\eta},\tilde{\lambda})\in\ci(\overline{\R}_+\times\overline{\R}_+\times\Omega,L^m_{\textup{cl}}(X;\R^{1+q+l}))$ there exists an $\tilde{h}(r,r',y,z,\tilde{\eta},\tilde{\lambda})\in \ci(\overline{\R}_+\times\overline{\R}_+\times\Omega,M^m_{\mathcal{O}}(X;\R^{q+l})),$ such that for every $\beta\in\R$ $$\textup{Op}_{r,y}(p)(\lambda)=\textup{Op}_y\textup{op}^\beta_M(h)(\lambda)\quad \textup{mod}\,L^{-\infty}(\R_+\times X\times\Omega;\R^l)$$ for $p(r,r',y,\rho,\eta,\lambda):=\tilde{p}(r,r',y,r\rho,r\eta,r\lambda),\,h(r,r',y,z,\eta,\lambda):=\tilde{h}(r,r',y,
z,r\eta,r\lambda).$
\end{Thm}
A proof may be found in \cite{Schu2}, see also \cite{Schu20}, or Krainer \cite{Krai2} for an alternative strategy.

\subsection{Edge operators with holomorphic symbols}
The first summand on the right of $\eqref{11.alg}$ (also referred to as the flat part of the edge calculus) consists of Mellin operators with $M^m_{\mathcal{O}}$-valued symbols, modulo certain Green operators appearing under compositions. The latter also occur in the asymptotic part of the edge calculus, then in combination with asymptotic quantities, cf. subsection 1.3 below. Although all those operators belong to $L^m_{\textup{deg}}(M;\R^l),$ apart from $\sigma_\psi$ they will possess a so-called edge symbol coming from the edge. The latter is the (twisted) homogeneous principal part of a certain classical operator-valued symbol. Let us  recall some generalities on such symbols. \\A Hilbert space $H$ is said to be endowed with a group action $\kappa=\{\kappa_\lambda\}_{\lambda\in\R_+}$ if $\kappa_\lambda:H\rightarrow H$ is a strongly continuous group of isomorphisms with
$\kappa_\lambda\kappa_\nu=\kappa_{\lambda\nu}$ for every $\lambda,\nu\in\R_+.$\\An example is the space
\begin{equation} \label{11.kegel}
\mathcal{K}^{s,\gamma}(X^\wedge):= \{u=\omega u_0+(1-\omega ) u_\infty :u_0\in\mathcal{H}^{s,\gamma}(X^\wedge), u_\infty\in H^s_{\textup{cone}}(X^\wedge)\}
\end{equation}
for $H^s_{\textup{cone}}(X^\wedge):=H^s_{\textup{cone}}(\R\times X)|_{\R_+\times X}$ and any cut-off function $\omega$ on the half-axis (i.e., $\omega(r)\in\ci_0(\overline{\R}_+),\omega=1$ in a neighbourhood of $0).$ In this case we set $(\kappa_\lambda u)(r,\cdot):=\lambda^{(n+1)/2}u(\lambda r,\cdot)$ for $n=\textup{dim}\,X.$\\A similar terminology will be used for a Fr\'echet space $E,$ written as a projective limit of Hilbert spaces $E^j,\,j\in\N,$ with continuous embeddings $E^j\hookrightarrow E^0$ for all $j;$ then we assume that $E^0$ is endowed with a group action that restricts to a group action on $E^j$ for every $j.$ An example is the space
\begin{equation} \label{11.kegelflat}
\mathcal{K}_\Theta^{s,\gamma}(X^\wedge):= \varprojlim_{m\in\N}\mathcal{K}^{s,\gamma+\theta-(m+1)^{-1}}(X^\wedge)
\end{equation}
where $\Theta=(\theta,0]\subseteq\overline{\R}_-$ is a weight interval.
\\If $H$ and $\tilde{H}$ are Hilbert spaces with group actions $\kappa$ and $\tilde{\kappa}$, respectively, an $\mathcal{L}(H,\tilde{H})$-valued function $f(y,\eta)$ on $U\times(\R^q\setminus\{0\}), U\subseteq\R^d$ open, is called (twisted) homogeneous in $\eta$ of order $\nu$ if$$f(y,\lambda\eta)=\lambda^\nu\tilde{\kappa}^{-1}_\lambda f(y,\eta)\kappa_\lambda$$ for all $\lambda\in\R_+.$ By $S^m(U\times\R^q;H,\tilde{H}),m\in\R,$ we denote the set of all smooth $\mathcal{L}(H,\tilde{H})$-valued functions $a(y,\eta)$ on $\Omega\times\R^q$ such that $$\textup{sup}_{(y,\eta)\in K\times\R^q}\langle\eta\rangle^{-m+|\beta|}||\tilde{\kappa}^{-1}_{\langle\eta\rangle}\{D^\alpha_yD^\beta_\eta a(y,\eta)\}\kappa_{\langle\eta\rangle}||_{\mathcal{L}(H,\tilde{H})}<\infty$$ for every $K\Subset U, \alpha\in\N^d,\beta\in\N^q.$ An $a\in S^m$ is said to be classical,
written $a\in S^m_{\textup{cl}},$ if there are homogeneous components $a_{(m-j)}(y,\eta)$ of homogeneity $m-j,\,j\in\N,$ such that $a=\chi\sum_{j=0}^Na_{(m-j)}\,\textup{mod}\,S^{m-(N+1)}$ for every $N\in\N$ and an excision function $\chi(\eta).$ A similar notation will be used for Fr\'echet spaces with group action. Clearly for $H=\tilde{H}=\C$ and $\kappa_\lambda=\tilde{\kappa}_\lambda=\textup{id}$ for all $\lambda$ we just recover the spaces $S^m_{(\textup{cl})}(\Omega\times\R^q)$ of scalar symbols. (We write subscript $``(\textup{cl})"$ if a statement is valid both in the classical and the general case). If necessary the dependence of the symbol spaces of the choice of $\kappa,\tilde{\kappa}$ will be indicated by corresponding subscripts, i.e., $S^m(\cdots)_{\kappa,\tilde{\kappa}}.$ Let us now turn to weighted edge spaces, first in abstract form, namely, $\mathcal{W}^s(\R^q,H),$ for a Hilbert space $H$ with group action $\kappa,$ defined to be the completion of $\mathcal{S}(\R^q,H)$ with respect to the norm
$||\langle\eta\rangle^s\kappa^{-1}_{\langle\eta\rangle}\hat{u}(\eta)||_{L^2(\R^q,H)},\,\hat{u}(\eta)=F_{y\rightarrow\eta}u(\eta).$ In a similar manner we define corresponding edge spaces for a Fr\'echet space $E$ with group action in place of $H.$ The choice of the group action will be clear from the context; if necessary we indicate $\kappa$ by a subscript. The case $\kappa_\lambda=\textup{id}$ for all $\lambda$ is admitted as well; then $\mathcal{W}^s(\R^q,H)_{\textup{id}}=H^s(\R^q,H)$ which is the standard Sobolev space of $H$-valued distributions of smoothness $s$. In the case of an open stretched wedge $M=X^\wedge\times\R^q$ for a closed and smooth manifold $X$ we form $$\mathcal{W}^s(\R^q,\mathcal{K}^{s,\gamma}(X^\wedge))$$ for $s,\gamma\in\R.$ In general, for a manifold $M$ with edge $Y$ the space $H^{s,\gamma}(M)$ is the subspace of those $u\in H^s_{\textup{loc}}(M\setminus Y)$ such that in the variables $(r,x,y)\in X^\wedge\times\Omega$ from $\eqref{11.wedge}$ we have $\omega\varphi u\in\mathcal{W}^s(\R^q,\mathcal{K}^{s,\gamma}(X^\wedge))$ for any cut-off function $\omega$ on the half-axis and $\varphi\in\ci_0(\Omega).$ This and the future global constructions refer to a fixed choice of local representations $\eqref{11.wedge}$ where we assume, for simplicity, that the transition maps are independent of $r$ for small $r.$ We might impose much weaker conditions; in any case, using a straightforward invariance property we obtain global spaces $H^{s,\gamma}(M)$ on a compact manifold $M$ with edge. Let us point out that the choice of local models $\eqref{11.wedge}$ can be interpreted as a kind of ``regular singular structure" formulated in terms of a specified system of singular charts where the local models are wedges and the analytic objects invariantly defined with respect to the cocycle of transition maps. In the case of (not necessarily compact) $M$ we write $H^{s,\gamma}_{[{\textup{loc}})}(M)$ where $``[\cdots)"$ indicates the specified behaviour close to $\partial\,\M.$ Moreover, $H^{s,\gamma}_{[{\textup{comp}})}(M)$ denotes the subspace of  elements such that the closure of the support  is compact in $\M.$ We also have flat wedge spaces $\mathcal{W}^s(\R^q,\mathcal{K}_\Theta^{s,\gamma}(X^\wedge)),$ for $\Theta=(\theta,0]$ as in $\eqref{11.kegelflat},$ and corresponding global spaces $H_\Theta^{s,\gamma}(M)$ when $M$ is compact, otherwise in the corresponding $``[\textup{loc})"$- or $``[\textup{comp})"$-versions. We will call an operator $C\in L^{-\infty}(\textup{int}\,M)$ smoothing and flat (with respect to the weight data ${\bf{g}}_{\gamma,\mu}:=(\gamma,\gamma-\mu,\Theta)$) if $C$ and its formal adjoint with respect to the $H^{0,0}(M)$-scalar product induce continuous operators $$C:H^{s,\gamma}(M)\rightarrow\ H_\Theta^{\infty,\gamma-\mu}(M),\,C^*:H^{s,-\gamma+\mu}(M)\rightarrow H_\Theta^{\infty,-\gamma}(M)$$ for all $s\in\R.$ In the present subsection we are interested in the case $\Theta=(-\infty,0].$ Then the definition is independent of the weights. Let $L_{\mathcal{O}}^{-\infty}(M)$ denote the corresponding space of smoothing operators (it is Fr\'echet in a natural way), and set $L_{\mathcal{O}}^{-\infty}(M;\R^l):= \mathcal{S}(\R^l,L_{\mathcal{O}}^{-\infty}(M)).$\\In order to formulate the first summand on the right of $\eqref{11.alg}$ we first introduce flat Green operators of the edge calculus. Let us set $\mathcal{K}^{s,\gamma;g}(X^\wedge):=\langle r\rangle^{-g}\mathcal{K}^{s,\gamma}(X^\wedge)$ for $g,s,\gamma\in\R.$ This is also a Hilbert space with group action $\kappa,$ defined by the same expression as for $\mathcal{K}^{s,\gamma}(X^\wedge),$ and we have an antilinear duality between $\mathcal{K}^{s,\gamma;g}$ and $\mathcal{K}^{-s,-\gamma;-g}$ via the $\mathcal{K}^{0,0}$-scalar product. The space of flat Green symbols $\mathcal{R}_{G,\mathcal{O}}^m(U\times\R^q),$ for $m\in\R,\,U\subseteq\R^d$ open, is defined to be the set of all
\begin{equation} \label{12.Green}
g(y,\eta), g^*(y,\eta)\in\bigcap S_{\textup{cl}}^m(U\times\R^q;\mathcal{K}^{s,\gamma;g}(X^\wedge)),\mathcal{K}^{s',\gamma';g'}(X^\wedge)))
\end{equation}
where the intersection is taken over all $s,s',\gamma,\gamma',g,g'\in\R,$ and $\eqref{12.Green}$ means the respective condition both for $g$ and its $(y,\eta)$-wise formal adjoint $g^*$ with respect to the $\mathcal{K}^{0,0}$-scalar product. The dimensions $d$ and $q$ are completely independent. In particular, we have the space $\mathcal{R}_{G,\mathcal{O}}^m(\Omega\times\R_{\eta,\lambda}^{q+l})$ of parameter-dependent Green symbols $g(y,\eta,\lambda)$ for $\Omega\subseteq\R^q$ open, and $\eta\in\R^q,\lambda\in\R^l.$ \\In the sequel a cut-off function $\sigma$ on $M$ near $Y$ is an element of $\ci(\textup{int}\,M)$ belonging to $\ci(\M)$ (under the identification $\textup{int}\,M \cong\textup{int}\,\M$) and being $1$ in a collar neighbourhood of $\partial\M$ and vanishing outside another collar neighbourhood.\\Now let $M$ be a (first compact) manifold with edge $Y.$ Then
\begin{equation} \label{11.flat}
 L^m_{\textup{flat}}(M,{\bf{g}}_{\gamma,\mu};\R^l)
\end{equation}
 for ${\bf{g}}_{\gamma,\mu}:=(\gamma,\gamma-\mu,\Theta)$ and $m=\mu-j$ is defined to be the subspace of all operator families $A(\lambda)\in L^m_{\textup{cl}}(M\setminus Y,;\R^l)$ which are of the form $$A_{\textup{edge}}+A_{\textup{int}}+C$$ for $A_{\textup{int}}\in L^m_{\textup{cl}}(M\setminus Y;\R^l)$ vanishing near $Y$ (i.e., $\sigma A_{\textup{int}}\sigma'=0$ for suitable cut-off functions $\sigma, \sigma'$ near $Y),$ moreover, $C\in L_{\mathcal{O}}^{-\infty}(M),$ and $A_{\textup{edge}}$ is (up to pull backs to a neighbourhood of the edge on $M$) a finite sum of operators \begin{equation} \label{1.quant}
 \Op_y(a)(\lambda ) \quad \mbox{for} \quad a(y,\eta ,\lambda ):=\epsilon\{r^{-m}\omega \op^{\gamma-n/2}_M(h)(y,\eta,\lambda)\omega'+g(y,\eta,\lambda)\}\epsilon '
\end{equation}
for Mellin symbols $h(r,y,z,\eta,\lambda):=\tilde{h}(r,y,
z,r\eta,r\lambda),\,\tilde{h}(r,y,z,\tilde{\eta},\tilde{\lambda})\in \ci(\overline{\R}_+\times \Omega,M^m_{\mathcal{O}}(X;\R^{q+l}_{\tilde{\eta},\tilde{\lambda}})),$ cut-off functions $\omega,\omega', \epsilon , \epsilon '$ on the half axis, and $g(y,\eta,\lambda)\in\mathcal{R}_{G,\mathcal{O}}^m(\Omega\times\R_{\eta,\lambda}^{q+l})$. The open sets $\Omega\subseteq\R^q$ correspond to charts on $Y,$ and the representations refer to the local variables from $\eqref{11.wedge}.$ The definition of flat edge operators implies $L^m_{\textup{deg}}(M;\R^l)\subseteq L^m_{\textup{flat}}(M,{\bf{g}}_{\gamma,\mu};\R^l).$ This gives us $\sigma _0(A)$ for every $A$ in this space. 
Recall that there is another equivalent definition of edge amplitude functions rather than $\eqref{1.quant}$, namely,
\begin{equation} \label{1.edamp}
a(y,\eta,\lambda ):= \epsilon\{r^{-m}\omega_{\eta,\lambda }\op^{\gamma -n//2}_M(h)(y,\eta,\lambda )\omega'_{\eta,\lambda }+\chi_{\eta,\lambda } \Op_r(p)(y,\eta,\lambda )\chi'_{\eta,\lambda }\}\epsilon'+g(y,\eta,\lambda );
\end{equation}
where $p$ and $h$ are associated via Theorem \ref{1.ph}.
\\Let $g_{(m)}$ denote the homogeneous principal component of $g,$ and set $h_0(r,y,z,\eta,\lambda)=\tilde{h}(0,y,z,r\eta,r\lambda).$ The family of operators  $$\sigma_1(A)(y,\eta,\lambda):= r^{-m}\op_M^{\gamma-n/2}(h_0)(y,\eta,\lambda)+g_{(m)}(y,\eta,\lambda)$$ for $(y,\eta,\lambda)\in T^*Y\times\R^l\setminus0$ (where $0$ means $(\eta,\lambda)=0)$ is an analogue of $\eqref{sy1},$ called the homogeneous principal edge symbol of $A$ of order $m.$
The ellipticity will refer to both symbolic components and the case $m=\mu.$ An $A\in L^\mu_{\textup{flat}}(M,{\bf{g}}_{\gamma,\mu};\R^l)$ for ${\bf{g}}_{\gamma,\mu}:=(\gamma,\gamma-\mu,\Theta)$ is called (parameter-dependent) $\sigma_0$-elliptic if $\sigma_0(A)$ never vanishes on $T^*(M\setminus Y)\times \R^l\setminus0,$ and if close to $Y$ the reduced symbol $\tilde{\sigma}_0(A)$ does not vanish for $(\tilde{\rho},\xi,\tilde{\eta},\tilde{\lambda})\neq0$ up to $r=0.$ Moreover, $A$ is called (parameter-dependent) elliptic if in addition 
\begin{equation} \label{ell}
\sigma_1(A)(y,\eta,\lambda):\mathcal{K}^{s,\gamma}(X^\wedge)\rightarrow\mathcal{K}^{s-\mu,\gamma-\mu}(X^\wedge)
\end{equation}
is a family of isomorphisms for all $(y,\eta,\lambda)\in T^*Y\times\R^l\setminus0.$ Observe that the ellipticity refers to a fixed weight $\gamma.$ The operators in $L^\mu_{\textup{flat}}(M,{\bf{g}}_{\gamma,\mu};\R^l)$
extend (or restrict) to corresponding operators in $L^\mu_{\textup{flat}}(M,{\bf{g}}_{\gamma',\mu};\R^l)$ for any $\gamma'\leqslant\gamma$ (or $\gamma'\geqslant\gamma),$ and it may happen that an elliptic $A$ with respect to $\gamma$ is not elliptic with respect to $\gamma'.$ For our purposes we need ellipticity in a prescribed weight interval, together with a corresponding order-reducing property. 
\begin{Thm}
For every $\mu\in\R$ and every $c\leqslant c'$ there exists an element an $A\in L^\mu_{\textup{flat}}(M,{\bf{g}}_{\gamma,\mu};\R^l)$
that is elliptic with respect to all weights $c\leqslant\gamma\leqslant c'$ and induces isomorphisms $\eqref{11.cont}$ for all $s\in\R,\lambda\in\R^l,c\leqslant\gamma\leqslant c'.$
\end{Thm}
The proof follows from the general techniques of the parameter-dependent edge calculus.

\subsection{The asymptotic part of the edge calculus}
\label{cylsecleib}
The second summand on the right of $\eqref{11.alg}$ reflects asymptotic phenomena in the solvability of elliptic equations on a manifold with edge. Operators in $L^m_{\textup{as}}(M,{\bf{g}}_{\gamma,\mu};\R^l)$ appear in parametrices of elliptic elements of $L^m_{\textup{flat}}(M,{\bf{g}}_{\gamma,\mu};\R^l)$ (in the case $m=\mu$). Then $L^\mu(M,{\bf{g}}_{\gamma,\mu};\R^l)$ itself turns out to be closed under the construction of parametrices of elliptic elements, provided that we employ a sufficiently flexible notion of asymptotics such as continuous or variable discrete asymptotics. Here we refer to continuous asymptotics. In principle we also might take the more refined notion of variable discrete asymptotics, cf. \cite{Schu55}, but we are mainly interested in the role of parameter-dependent edge operators for the corner calculus; therefore, we try to minimize the effort for the technical details and single out a structure which controls asymptotic data in the weight strips $\{z\in\C:(n+1)/2-\gamma -1< \textup{Re}\,z< (n+1)/2-\gamma \}$ and $\{z\in\C:(n+1)/2-(\gamma +\mu)-1<\textup{Re}\,z< (n+1)/2-(\gamma +\mu)\},$ respectively, for $n=\textup{dim}\,X.$ In other words, throughout this exposition for convenience in  ${\bf{g}}_{\gamma,\mu}:=(\gamma,\gamma-\mu,\Theta)$ we set $\Theta=(-1,0].$\\
By a continuous Mellin asymptotic type (for principal Mellin symbols) we understand a subset $R\subset\C$ such that $R\cap\Gamma_\beta=\emptyset$ for some $\beta\in\R$ and $R\cap \{z\in\C:c\leq\textup{Re}\,z\leq c'\}$ compact for every $c\leq c'.$ Let $M^{-\infty}_R(X)$ be the subspace of all $f(z)\in \mathcal{A}(\C\setminus R,L^{-\infty}(X))$ such that for every $R$-excision function $\chi(z)$ we have $\chi f|_{\Gamma_\beta}\in L^{-\infty}(X;\Gamma_\beta)$ for every $\beta\in\R,$ uniformly in compact $\beta$-intervals. In the case $R=\emptyset$ we simply write $M^{-\infty}_{\mathcal{O}}(X).$ \\Moreover, a continuous asymptotic type (for spaces on $X^\wedge $ with weight) is a subset $P\subset\{z\in\C:\textup{Re}\,z< \delta\}$ for some real $\delta$ such that $P\cap \{z\in\C:c\leq\textup{Re}\,z\leq c'\}$ is compact for every $c\leq c'.$ In the case of spaces $\mathcal{K}^{s,\gamma}(X^\wedge), n=\textup{dim}\,X,$ we set $\delta=(n+1)/2-\gamma.$ Now the subspace $\mathcal{K}_P^{s,\gamma}(X^\wedge)$ with asymptotics of type $P$ is defined to be the non-direct sum
\begin{equation} \label{1.as}
\mathcal{K}_P^{s,\gamma}(X^\wedge):=\mathcal{K}_\Theta^{s,\gamma}(X^\wedge)+\mathcal{E}_P
\end{equation} for
$\mathcal{E}_P:=\{\omega(r)M_{w\rightarrow r}^{-1}\langle\zeta_z,\Phi(z,w)\rangle: \zeta\in\mathcal{A}'(P_\varepsilon,C^\infty(X))\},$
where $P_\varepsilon:=\{z\in P:\textup{Re}\,z\geq (n+1)/2-\gamma-1-\varepsilon\},$ and $\mathcal{A}'(K,C^\infty(X))$ is the set of all $C^\infty(X)$-valued analytic functionals carried by the compact set $K.$ Moreover, $\Phi(z,w):=M_{r\rightarrow z}(r^{-w}\omega(r))$ for some cut-off function $\omega.$  The space $\eqref{1.as}$ is Fr\'echet in a natural way and independent of the choice of $\varepsilon.$  More generally, we set $\mathcal{K}_P^{s,\gamma;g}(X^\wedge):=\langle r\rangle^{-g}\mathcal{K}_P^{s,\gamma}(X^\wedge)$ for any real number $g.$ For any weight $\gamma\in\R,$ and a Mellin symbol  $f(y,z)\in \ci(\Omega,M_R^{-\infty}(X))$ such that 
$R \cap \Gamma_{(n+1)/2- \gamma}=\emptyset,$ and cut-off functions $\omega, \omega',$ 
for $\omega_\eta(r)=\omega(r[\eta]),$ etc., we have a family of smoothing Mellin operators $m(y,\eta):=r^{-\mu}\omega_\eta\op_M^{\gamma-n/2}(f)(y)\omega'_\eta;$ those induce continuous operators 
\begin{equation}
m(y,\eta):\mathcal{K}^{s,\gamma}(X^\wedge)\rightarrow\mathcal{K}^{\infty,\gamma-\mu}(X^\wedge), \,\mathcal{K}_P^{s,\gamma}(X^\wedge)\rightarrow\mathcal{K}_Q^{\infty,\gamma-\mu}(X^\wedge)
\end{equation}
for every $s\in\R$ and every asymptotic type $P$ for some resulting asymptotic type $Q.$ Recall that 
\begin{equation}
m(y,\eta)\in S_{\textup{cl}}^\mu(\Omega\times\R^q;\mathcal{K}_{(P)}^{s,\gamma;g}(X^\wedge),\mathcal{K}_{(Q)}^{s-\mu,\gamma-\mu;g'}(X^\wedge))
\end{equation}
for all $s,g,g'\in\R;$ parentheses at the asymptotic types mean spaces with or without asymptotics on both sides. Note that smoothing Mellin operators play a role in many contexts of analysis on manifolds with conical singularities and edges, see, in mparticular, \cite{Remp1} in the case of boundary value problems, using previous work of Eskin \cite{Eski2}.\\
Another ingredient of the asymptotic part of the edge calculus are Green symbols with asymptotics. The space of Green symbols $\mathcal{R}_G^m(\Omega\times\R^q,{\bf{g}}_{\gamma,\mu}),$ for $m\in\R,\,\Omega\subseteq\R^q$ open, and asymptotic types $P,Q,$ is defined to be the set of all operator functions $g(y,\eta)$ such that 
\begin{equation} \label{12.Greenas}
g(y,\eta)\in\bigcap S_{\textup{cl}}^m(\Omega\times\R^q;\mathcal{K}^{s,\gamma;g}(X^\wedge)),\mathcal{K}_P^{s',\gamma-\mu;g'}(X^\wedge))),
\end{equation}
and
\begin{equation} \label{12.Greenstar}
 g^*(y,\eta)\in\bigcap S_{\textup{cl}}^m(\Omega\times\R^q;\mathcal{K}^{s,-\gamma+\mu;g}(X^\wedge)),\mathcal{K}_Q^{s',-\gamma;g'}(X^\wedge)))
\end{equation}
where the intersection is taken over all $s,s',g,g'\in\R,$ and $\eqref{12.Green}$ means the respective condition both for $g$ and its $(y,\eta)$-wise formal adjoint $g^*$ with respect to the $\mathcal{K}^{0,0}$-scalar product. In the following  we employ the symbols $m$ and $g$ in the version with covariables $\eta,\lambda$ rather than $\eta$ (the modification is straightforward). The spaces $\mathcal{K}_{P}^{s,\gamma;g}(X^\wedge)$  admit the group action $\kappa =\{\kappa _\lambda \}_{\lambda \in \R_+};$ thus we can form wedge spaces with asymptotics $\mathcal{W}^s(\R^q,\mathcal{K}_{P}^{s,\gamma;g}(X^\wedge)).$ Those give rise to corresponding global spaces $H_P^{s,\gamma}(M) $ on a compact manifold $M$ with edge (otherwise, in the non-compact case, corresponding comp- or loc- variants).
Let  $L^{-\infty}_{\textup{as}}(M,{\bf{g}}_{\gamma,\mu})$ denote the set of all $C:H^{s,\gamma}(M)\rightarrow H^{s-\mu,\gamma-\mu}(M)$ which induce continuous operators
$$C:H^{s,\gamma}(M)\rightarrow H_P^{\infty,\gamma-\mu}(M),\,C^*:H^{s,-\gamma+\mu }(M)\rightarrow H_Q^{\infty,\gamma-\mu}(M)$$
where $C^*$ is the formal adjoint with respect to the $\mathcal{W}^{0,0}(M)$-scalar product; this is required for all $s\in\R$ and for certain asymptotic types $P$ and $Q$ depending on $C.$ Moreover, we set $L^{-\infty}_{\textup{as}}(M,{\bf{g}}_{\gamma,\mu};\R^l):= \mathcal{S}(\R^l,L^{-\infty}_{\textup{as}}(M,{\bf{g}}_{\gamma,\mu})).$
The space 
\begin{equation} \label{11.algas}
L^{\mu -j}_{\textup{as}}(M,{\bf{g}}_{\gamma,\mu};\R^l),\,j\in\N,
\end{equation} 
is defined to be the set of all $$A_{\textup{edge}}(\lambda )+C(\lambda )$$ where (up to pull backs to a neighbourhood of the edge on $M$) the operator$A_{\textup{edge}}(\lambda )$ is a finite sum of operators 
$\sigma \Op_y(a)(\lambda)\sigma '$ for cut-off fuctions $\sigma, \sigma ',$ and symbols  $a(y,\eta,\lambda):=m(y,\eta,\lambda)+g(y,\eta,\lambda),$ with $m$ and $g$ being of the above-mentioned form where $m$ vanishes for $j>0,$ and smoothing operators $C(\lambda)\in L^{-\infty}_{\textup{as}}(M,{\bf{g}}_{\gamma,\mu};\R^l).$ 
For an operator $A\in L^{\mu-j} _{\textup{as}}(M,{\bf{g}}_{\gamma,\mu};\R^l)$ we set
$\sigma_0(A):=0,$ $$\sigma _1(A)(y,\eta,\lambda  ):=r^{-\mu}\omega (r|\eta,\lambda  |)\op_M^{\gamma -n/2}(f)(y) \omega' (r|\eta,\lambda  |)+g_{(\mu )}(y,\eta,\lambda  )$$in the case $j=0$ and
$$\sigma _1(A)(y,\eta,\lambda  ):=g_{(\mu-j )}(y,\eta,\lambda  )$$ for $j>0$ where $g_{(\mu-j )}(y,\eta,\lambda  )$ is the homogeneous principal part of $g$ as a classical symbol of order $\mu-j $.\\ The spaces $\eqref{11.algas}$ together with $\eqref{11.flat}$ furnish the calculus of edge operators $\eqref{11.alg}.$ We do not recall here all details. 
In any case we have compositions in the sense that
\begin{equation} \label{11.compo}
A\in L^m(M,{\bf{g}}_{\gamma,\mu};\R^l),\,\tilde{A}\in L^{\tilde{m}}(M,{\bf{g}}_{\tilde{\gamma},\tilde{\mu}};\R^l)\Rightarrow A\tilde{A}\in L^{m+\tilde{m}}(M,{\bf{g}}_{\tilde{\gamma},\mu+\tilde{\mu}};\R^l), \sigma (A\tilde{A})=\sigma (A)\sigma (\tilde{A})
\end{equation}
for $\gamma =\tilde{\gamma}-\tilde{\mu};$ the composition of symbols is componentwise.\\ 
Let us finally note that the assumption on the length of the weight interval $\Theta $ allows us to ignore lower order smoothing Mellin symbols; the generalities of the edge calculus remain untouched.
\subsection{Ellipticity of edge operators}
A basic tool for the higher corner spaces are parameter-dependent elliptic operators on a compact base with edge. As noted in the very beginning, ellipticity of an operator is a condition on its principal symbols, in the present case, $\sigma=(\sigma_0,\sigma_1),$ here in parameter-dependent form. An element $A$ in $L^\mu (M,{\bf{g}}_{\gamma,\mu};\R^l)$ is called elliptic if first $A$ is $\sigma _0$-elliptic, i.e., $\sigma _0(A)$ never vanishes as a function on $T^*s_0(M)\times \R^l\setminus 0,$ and the reduced symbol $\tilde{\sigma }_0(A)$ does not vanish up to the boundary of the stretched manifold, and if in addition the principal edge symbol
\begin{equation} \label{11.prc}
\sigma _1(A)(y,\eta,\lambda  ): \mathcal{K}^{s,\gamma}(X^\wedge)\rightarrow \mathcal{K}^{s-\mu ,\gamma-\mu }(X^\wedge)
\end{equation}
is bijective for all $(y,\eta ,\lambda) \in T^*s_1(M)\times \R^l\setminus 0$ for some $s\in \R$ (the choice of $s$ in this condition is unessential). We may interprete $\sigma _1(A)(y,\eta,\lambda  )$ as a family of operators in the cone algebra on the (open stretched) cone $(X^\wedge),$ parametrised by the variables $y,\eta,\lambda .$Thus there is a subordinate ($(\eta,\lambda  )$)-independent) principal conormal symbol $\sigma _1(\sigma _1(A)(y,\eta,\lambda  ))=: f(y,z)$ with the complex covariable $z.$ The following result will be formulated, for simplicity, for the flat part of the edge calculus.
\begin{Thm} \label{1.prep}
Let $A(\lambda )\in L^\mu_{\textup{flat}} (M,{\bf{g}}_{\gamma,\mu};\R^l)$ be a $\sigma _0$-elliptic operator. Then for every $y\in s_1(M)$ there is a discrete set $D(y)\subset \C$ intersecting $\{z\in \C:c\leq \textup{Re}\,z\leq c'\}$ in a finite set for every $c\leq c'$ such that
\begin{equation} \label{11.bijsob}
f(y,z):H^s(X)\rightarrow H^{s-\mu}(X) 
\end{equation}
are isomorphisms for all $z\in \C\setminus D(y)$ and all $s\in \R.$
\end{Thm}
 Since the $\sigma _0$-ellipticity of $A$ entails the $\sigma _0$-ellipticity of $\sigma _1(A)(y,\eta,\lambda  )$ as an operator in the cone algebra over 
$X^\wedge$ (including the exit-ellipticity for $r\rightarrow \infty $) we have the Fredholm property of $\eqref{11.prc}$ only for those $\gamma \in \R$ such that $\Gamma _{(n+1)/2-\gamma }\cap D(y)=\emptyset.$
If we insist on the Fredholm property for a particular weight $\gamma _1,$ according to \cite{Mali1} we can add to $A$ a smoothing Mellin operator $M$ to obtain for $A_1:=A+M$ in place of $A$ a Fredholm family $\eqref{11.prc}$ for $\gamma =\gamma _1.$What concerns the general structures we concentrate on the case of the bijectivity, provided that a topological obstruction (similarly as a corresponding obstruction in boundary value problems, cf. Atiyah and Bott \cite{Atiy5} ) is vanishing. Since the consequences will be crucial here we briefly recall the main points. First we have the following homogeneity relation
\begin{equation} \label{11.prhom}
\sigma _1(A)(y,\delta \eta,\delta \lambda  )=\delta^\mu   \kappa_ \delta^{-1} \sigma _1(A)(y, \eta,\lambda  )\kappa_ \delta ,\,\delta \in \R_+. 
\end{equation}
Thus, for $S:=\{(y, \eta,\lambda )\in T^*s_1(M)\times \R^l, |\eta ,\lambda |=1\}$ with the canonical projection $\pi :S\rightarrow s_1(M),$ the restriction of $\eqref{11.prc}$ to $S$ (for brevity denoted again by $\sigma _1(A)$) allows us to recover $\eqref{11.prc}$ in a unique way, but the restriction is a Fredholm family parametrised by the compact set $S.$ This gives us an element in the $K$-group over $S$ in a well-known manner, namely, $\textup{ind}_{S}\,\sigma _1(A)\in K(S).$ Now if $\eqref{11.prc}$ only consists of Fredholm operators rather than isomorphisms, the idea is (similarly as in boundary value problems) to fill up the mapping by finite rank operators to a $2\times 2$ block matrix of isomorphisms. In boundary value problems this construction is known to generate the symbols of additional conditions of trace and potential type. The same is the case in edge problems. However, the extra $(y, \eta,\lambda )$-depending operator families cannot always be interpreted as such symbols. A necessary and sufficient condition for that is the relation
\begin{equation} \label{11.obstr}
\textup{ind}_{S}\,\sigma _1(A)\in K(S)\in \pi ^*K(s_1(M))
\end{equation}
where $\pi ^*$ is the homomorphism between the respective $K$-groups, induced by the bundle pull back. There are non-trivial examples where $\eqref{11.obstr}$ is violated, 
cf. Nazaikinskij, Savin, Schulze, Sternin \cite{Naza6}. Note that in such a case a Fredholm theory of edge problems can always be organised in terms of global projection conditions, see Schulze, Seiler \cite{Schu42}, analogously as the well-known projection conditions of APS-type. Throughout this paper we assume that $\eqref{11.obstr}$ is satisfied. It will be adequate in this case to employ extra smoothing terms from $L^\mu _{\textup{as}}(M,{\bf{g}}_{\gamma,\mu};\R^l)$ to achieve isomorphisms. 
\begin{Thm} \label{bje}
Let $A(\lambda )\in L^\mu (M,{\bf{g}}_{\gamma,\mu};\R^l)$ be a $\sigma _0$-elliptic operator, and let $\gamma \in \R$ such that $\eqref{11.prc}$ is a family of Fredholm operators. Then there exists an $(M+G)(\lambda )\in L^\mu_{\textup{as}}(M,{\bf{g}}_{\gamma,\mu};\R^l)$ such that
\begin{equation} \label{11.bij}
\sigma _1(A+M+G)(y,\eta,\lambda  ): \mathcal{K}^{s,\gamma}(X^\wedge)\rightarrow \mathcal{K}^{s-\mu ,\gamma-\mu }(X^\wedge)
\end{equation}
is a family of isomorphisms.
\end{Thm}
\begin{proof}
By assumption we have $\textup{ind}_{S}\,\sigma _1(A)= [E_+]-[E_-]$ for certain smooth complex vector bundles $E_+,E_-$ over $s_1(M).$ First observe that for every $j\in \Z$ there exists a  smoothing Mellin
operator $\op_M(f): L^2(\R_+)\rightarrow L^2(\R_+)$ for a symbol $f(z)\in S^{-\infty }_{\textup{cl}}(\Gamma _{1/2})$ which extends to a function $h\in \mathcal{A}(\C)$ with $h|_{\Gamma_{1/2}}=f|_{\Gamma_{1/2}}$ and $f(z)|_{\Gamma_{1/2-\beta }}\in S^{-\infty }_{\textup{cl}}(\Gamma _{1/2-\beta })$ for every $\beta \in \R,$ such that
$\textup{ind}\,\big(1+\omega \op_M(f)\omega \big)=j.$ This can be used to construct an element $f\in \ci(s_1(M),M_{\mathcal{O}}^{-\infty }(X))$ and an operator $B:=1+\omega_{\eta,\lambda }\op_M^{\gamma -n/2}(f) \omega_{\eta,\lambda },\,\omega_{\eta,\lambda }=\omega (r[\eta,\lambda]),$ such that $ \textup{ind}_{S}\,\sigma _1(B )=[E_-]-[E_+]. $ The expression for $B$ is an abbreviation for a corresponding operator which is first of such a form locally close to $s_1(M)$ and then defined globally by applying a partition of unity. The composition $AB$ belongs to $L^\mu (M,{\bf{g}}_{\gamma,\mu};\R^l),$ and we have $\textup{ind}_{S}\,\sigma _1(AB)=0.$ Now there is a flat Green operator $G\in L^\mu (M,{\bf{g}}_{\gamma,\mu};\R^l),$ i.e., with local symbols in $\mathcal{R}_{G,\mathcal{O}}^\mu (\Omega\times\R_{\eta,\lambda}^{q+l})$ (taking values in finite rank operators) such that $AB+G\in L^\mu (M,{\bf{g}}_{\gamma,\mu};\R^l)$ has a bijective edge symbol
$\sigma _1(AB+G).$ 
\end{proof}
As a consequence of the latter construction we have $\sigma _0(AB+G)=\sigma _0(A)$ and $AB+G=A \,\,\mbox{mod}\,L_{\textup{as}}^\mu (M,{\bf{g}}_{\gamma,\mu};\R^l). $

\section{Operators on an infinite singular cone}

\subsection{Corner-degenerate operators}
Let $B$ be a compact manifold with smooth edge $s_1(B),$ and consider the infinite cone $B^\Delta=(\overline{\R}_+\times B)/(\{0\}\times B),$ or its stretched version $B^\wedge=\R_+\times B$ in the splitting of variables $(t,b).$ The vertex $\{c\}$ of $B^\Delta$ is a singularity of second order (in our terminology), while $B^\wedge=B^\Delta\setminus \{c\}$ is a manifold with smooth edge $s_1(B)^\wedge.$ Algebras of pseudo-differential operators on cones $B^\Delta$ close to $c$ are studied in \cite{Schu27}, including ellipticity and (iterated) asymptotics of solutions in spaces with double weights. Elements of the calculus for $t\rightarrow\infty$ are developed in \cite{Calv2} and \cite{Calv3}. Similarly as edge symbols $\sigma_1$ associated with $s_1(B),$ operating in $\mathcal{K}^{s,\gamma}(X^\wedge)$-spaces, $s,\gamma\in\R,$ Fredholm in the elliptic case, on a wedge $W:=B^\Delta\times\Sigma,\,\Sigma\subseteq\R^d$ open, we should study edge symbols $\sigma_2$ associated with $s_2(W)=\Sigma,$ acting in spaces of the kind $\mathcal{K}^{s,\gamma,\theta}(B^\wedge),$ for $s,\gamma,\theta \in\R.$ It becomes increasingly difficult to explicitly control the data that are involved in the ellipticity and the Fredholm property. Here we construct new classes of elliptic operators. To illustrate the situation we first consider corner-degenerate differential operators. On the manifold $B$ with edge we have the space $\mbox{Diff}^m_{\textup{deg}}(B)$ of all $A\in\mbox{Diff}^m(s_0(B))$ that are close to $s_1(B)$ in the local variables $(r,x,y)\in X^\wedge\times \Omega$ of the form $\eqref{ecE}.$ Now on the wedge $W$ we have the space $\mbox{Diff}^m_{\textup{deg}}(W)$ of all $A\in \mbox{Diff}^m_{\textup{deg}}(B^\wedge\times\Sigma)$ (taking into account that $B^\wedge\times\Sigma$ is a manifold with edge) which are in the variables $(t,b,v)\in B^\wedge\times \Sigma$ of the form 
$$A=t^{-m}\!\!\sum_{j+|\alpha |\leq m}a_{j\alpha }(t,v)( -t\partial_t)^j(tD_v)^\alpha $$ 
with coefficients $a_{j\alpha}\in \ci(\overline{\R}_+\times\Sigma,\mbox{Diff}_{\textup{deg}}^{m-(j+|\alpha|)}(B)).$ Similarly as $\eqref{sy1}$
we have a homogeneous principal edge symbol
\begin{equation} \label{21.edge}
\sigma_1(A)(v,\zeta):=t^{-m}\!\!\sum_{j+|\alpha |\leq m}a_{j\alpha }(0,v)( -t\partial_t)^j(t\zeta)^\alpha\,:\mathcal{K}^{s,\gamma,\theta}(B^\wedge)\rightarrow\mathcal{K}^{s-m,\gamma-m,\theta-m}(B^\wedge), 
\end{equation}
acting between analogues of the $\mathcal{K}^{s,\gamma}$-spaces, but here with the extra weight $\theta$ belonging to the corner axis variable $t.$ We recall the definition of those spaces below (in a new form, compared with \cite{Calv2}); $\zeta$ is the covariable to $v\in\Sigma,$ and \eqref{21.edge} is considered for $(v,\zeta)\in\Sigma\times(\R^d\setminus\{0\}).$

\begin{Def} \label{2.mehol}
$\textup{(i)}$ Let $M^m_{\mathcal{O}}(B,\g_{\gamma,\mu};\R^l)$ denote the space of all $h(w,\lambda)\in\mathcal{A}(\C,L^m(B,\g_{\gamma,\mu};\R^l))$ such that $h(\delta+i\tau,\lambda)\in L^m(B,\g_{\gamma,\mu};\R^{1+l})$ for every $\delta\in\R,$ uniformly in compact $\delta$-intervals.\\\\
$\textup{(ii)}$ Let $R:=\{r_j\}_{j\in \Z}\subset \C$ be a set such that $R\cap \{w\in \C:c\leq c'\}$ is finite for every $c\leq c'.$ Then $M^{-\infty}_R (B,\g_{\gamma,\mu})$ denotes the set of all $h(w)\in\mathcal{A}(\C\setminus R,L^{-\infty }(B,\g_{\gamma,\mu}))$ such that $(\chi h)(\delta+i\tau)\in L^{-\infty }(B,\g_{\gamma,\mu};\R)$ for every $\delta\in\R,$ uniformly in compact $\delta$-intervals, for any $R$-excision function $\chi .$\\\\
$\textup{(iii)}$ We set $$M^m_R(B,\g_{\gamma,\mu}):=M^m_{\mathcal{O}}(B,\g_{\gamma,\mu})+M^{-\infty}_R (B,\g_{\gamma,\mu}).$$
\end{Def}
Given an element $a(w,\lambda)\in L^m(B,\g_{\gamma,\mu};\Gamma_0\times\R^l)$ we define a Mellin kernel cut-off operator
\begin{equation}         \label{cut}
V(\varphi)a(i\tau,\lambda):=\int_0^\infty s^{i\tau}\varphi(s)k(a)(s,\lambda)s^{-1}ds
\end{equation}
for $\varphi\in\ci_0(\R_+),$ and $k(a)(s,\lambda):=\int_{-\infty}^\infty s^{-i\tau}a(i\tau,\lambda)\dbar \tau.$
\begin{Thm} \label{2.conv}
The formula $\eqref{cut}$ defines an operator
\begin{equation}
V(\varphi):L^m(B,\g_{\gamma,\mu};\Gamma_0\times\R^l)\rightarrow M^m_{\mathcal{O}}(B,\g_{\gamma,\mu};\R^l),
\end{equation}
$V(\varphi ):f\rightarrow h,$ separately continuous in $\varphi$ and $f\in L^m.$ For every $\psi\in\ci_0(\R_+)$ that is equal to $1$ in a neighbourhood of $s=1$ we have 
\begin{equation}
f=V(\psi)f|_{\Gamma_0\times\R^l} \quad\textup{mod}\,L^{-\infty}(B,\g_{\gamma,\mu};\Gamma_0\times\R^l).
\end{equation}
Moreover, setting $\psi_{\varepsilon}(s)=\psi(\varepsilon \textup{log}\,s),\varepsilon >0,$ we have
$\textup{lim}_{\varepsilon\rightarrow 0}V(\psi_{\varepsilon})f|_{\Gamma_0\times\R^l}=f$ with convergence in $L^m.$
\end{Thm}

\begin{Thm} \label{2.me}
For every $\tilde{p}(t,t',\tilde{\tau},\tilde{\lambda})\in\ci(\overline{\R}_+\times\overline{\R}_+,L^m(B,\g_{\gamma,\mu};\R^{1+l}))$ there is an $\tilde{h}(t,t',w,\tilde{\lambda})\in \ci(\overline{\R}_+\times\overline{\R}_+,M^m_{\mathcal{O}}(B,\g_{\gamma,\mu};\R^l)),$ such that for every $\delta\in\R$ $$\textup{Op}_t(p)(\lambda)=\textup{op}^\delta_M(h)(\lambda)\quad \textup{mod}\,L^{-\infty}(B^\wedge,\g_{\gamma,\mu};\R^l)$$ for $p(t,t',\tau,\lambda):=\tilde{p}(t,t',t\tau,t\lambda),\,h(t,t',w,\lambda):=\tilde{h}(t,t',
w,t\lambda).$
\end{Thm}
A proof may be found in \cite{Schu2}, see also \cite{Schu20}, or \cite{Krai2}. For references below we set
\begin{equation} \label{21.null}
p_0(t,\tau,\lambda):=\tilde{p}(0,0,t\tau,t\lambda),\, h_0(t,w,\lambda):=\tilde{h}(0,0,w,t\lambda).
\end{equation}
\begin{Def} \label{meel}
An element $h(w) \in M^\mu _R(B,\g_{\gamma,\mu})$ is called elliptic if there is a $\delta \in \R, \Gamma_\delta  \cap \R=\emptyset,$ such that  $h(\delta+i\tau)$ is elliptic in $L^\mu (B,\g_{\gamma,\mu};\R_\tau ).$ 
\end{Def}
The following theorem may be found in \cite{Mani2}.
\begin{Thm} \label{2.meinv}
For every elliptic $h(w) \in M^\mu _R(B,\g_{\gamma,\mu})$ there exists an $h^{-1}(w) \in M^{-\mu} _S(B,\g_{\gamma-\mu ,-\mu })$ for a certain $S$ such that $hh^{-1}=1$ and $h^{-1}h=1.$
\end{Thm}
Similarly as $\eqref{11.kegel}$ we set
\begin{equation} \label{21.kegel}
\mathcal{K}^{s,\gamma,\theta}(B^\wedge ):= \{u=\omega u_0+(1-\omega )u_\infty :u_0\in\mathcal{H}^{s,\gamma,\theta}(B^\wedge ), u_\infty\in H^{s,\gamma}_{\textup{cone}}(B^\wedge)\}
\end{equation}
for any cut-off function $\omega(t).$ The spaces on the right hand side of $\eqref{21.kegel}$ are defined as follows. We use the fact that for every $\mu,\gamma\in\R$ there exists an elliptic family of operators 
\begin{equation} \label{21.red}
R^\mu(\lambda)\in L^\mu(B,{\bf{g}}_{\gamma,\mu};\R^l)
\end{equation}
which induces isomorphisms
\begin{equation} \label{21.iso}
R^\mu(\lambda): H^{s,\gamma}(B)\rightarrow H^{s-\mu,\gamma-\mu}(B)
\end{equation}
for all $s\in\R$ and $\lambda\in\R^l.$ The space $\mathcal{H}^{s,\gamma,\theta}(B^\wedge)$ is defined as the completion of $\ci_0(\R_+, H^{\infty ,\gamma}(B))$ with respect to the norm
\begin{equation} \label{21.weight}
\big\{\int_{\Gamma_{(\textup{dim}\,B+1)/2-\theta}}\|R^s(\textup{Im}\,w)M_{t \rightarrow w}f(w,\cdot)\|^2_{H^{0,\gamma-s}(B)}\dbar w\big\}^{1/2};
\end{equation}
where we employ $\eqref{21.red}$ for $l=1.$ Moreover, we set
$H^{s,\gamma}_{\textup{cone}}(B^\wedge):=H^{s,\gamma}_{\textup{cone}}(\R\times B)|_{B^\wedge}$ where $H^{s,\gamma}_{\textup{cone}}(\R\times B)$ is defined as the completion of $\ci_0(\R, \mathcal{W}^{\infty ,\gamma}(B))$ with respect to the norm
\begin{equation} \label{21.cone}\big\{\int\|\langle t\rangle^{-s+\textup{dim}\,B/2}\big(F^{-1}_{\tau\rightarrow t}R^s(\langle t\rangle\tau,\langle t\rangle\zeta)(F_{t\rightarrow\tau}u)\big)(t,\cdot)\|^2_{H^{0,\gamma-s}(B)}dt\big\}^{1/2},
\end{equation}
for any fixed $\zeta\in\R^d$ of sufficiently large absolute value; here we employ $\eqref{21.red}$ for $l=1+d.$ Note that the spaces $H^{s,\gamma}_{\textup{cone}}(B^\wedge)$ generalise $H^s_{\textup{cone}}(M^\wedge)$ mentioned in Section 1.1 for a smooth compact manifold $M.$ We do not focus our considerations here to what is also called manifolds with conical exits to infinity. For basics in the smooth case, see Shubin \cite{Shub1}, Parenti \cite{Pare1}, or Cordes \cite{Cord1}. In or case $B^\wedge$ is treated as a manifold with edge and conical exit to infinity where the cross section $B$ has edges, see also \cite{Calv1}.

\subsection{Parameter-dependent operators in corner spaces}
Pseudo-differential operators on a wedge $B^\Delta\times\Sigma$ for a manifold $B$ with edge and an open set $\Sigma\subseteq\R^d$ contain amplitude functions of the form $t^{-\mu}p(t,v,\tau,\zeta)$ for
$$p(t,v,\tau,\zeta):=\tilde{p}(t,v,t\tau,t\zeta),\quad   
\tilde{p}(t,v,\tilde{\tau},\tilde{\zeta})\in L^\mu(B,{\bf{g}}_{\gamma,\mu};\R^{1+d}_{\tilde{\tau},\tilde{\zeta}}),$$ for variables $(t,v)\in\R_+\times\Sigma.$ Together with associated Mellin amplitude functions $h(t,v,w,\zeta),$ cf. Theorem \ref{2.me}, those are involved in $L^\mu(B^\wedge,{\bf{g}}_{\gamma,\mu})$-valued edge symbols, namely,
\begin{equation} \label{2.edamp}
a(v,\zeta):= t^{-\mu}\epsilon\{\omega_\zeta\op^{\theta-\textup{dim}\,B/2}_M(h)(v,\zeta)\omega'_\zeta+\chi_\zeta \Op_t(p)(v,\zeta)\chi'_\zeta\}\epsilon'+(m+g)(v,\zeta)+b(v,\zeta);
\end{equation}
here $\epsilon, \epsilon', \omega, \omega'$ are cut-off functions and $\chi, \chi'$ excision functions in $t;$ moreover, $\omega_\zeta(t):=\omega(t[\zeta]),$ etc. 
The summand $b(v,\zeta)$ is localised far from $t=0,$ and  $(m+g)(v,\zeta)$ are smoothing Mellin plus Green contributions
of a similar structure as those in the edge calculus for smooth edges. We have
$$a(v,\zeta)\in S^\mu(\Sigma\times\R^d;\mathcal{K}^{s,\gamma,\theta}(B^\wedge),\mathcal{K}^{s-\mu,\gamma-\mu,\theta-\mu}(B^\wedge))$$ for every $s\in\R.$ Let us form the principal edge symbol
\begin{equation}  \label{2.edprinc}
\sigma_2(a)(v,\zeta):= t^{-\mu}\!\{\omega(t|\zeta|)\op^{\theta-\textup{dim}\,B/2}_M\!(h_0)(v,\zeta)\omega'(t|\zeta|)+\chi(t|\zeta|)\Op_t(p_0)(v,\zeta)\chi'(t|\zeta|)\!\}\!+\sigma_2(m+g)(v,\zeta),
\end{equation}\\
cf., $\eqref{21.null},$
\begin{equation} \label{2.edsy}
\sigma_2(a)(v,\zeta):\mathcal{K}^{s,\gamma,\theta}(B^\wedge)\rightarrow\mathcal{K}^{s-\mu,\gamma-\mu,\theta-\mu}(B^\wedge).
\end{equation}
Similarly as in first order principal symbols we have homogeneity in the sense
\begin{equation} \label{2.khom}
\sigma_2(a)(v,\lambda \zeta)=\lambda ^\mu \kappa _\lambda \sigma_2(a)(v, \zeta) \kappa _\lambda^{-1}
\end{equation}
for all $\lambda \in \R_+.$ Here $(\kappa _\lambda u)(t)=\lambda ^{(\textup{dim}\,B+1)/2}u(\lambda t)$ for $u(t)\in \mathcal{K}^{s,\gamma,\theta}(B^\wedge).$
As noted at the beginning the ellipticity of the 
operator $\Op(a)$ contains the condition that $\eqref{2.edsy}$ is a family
of isomorphisms for all $(v,\zeta)\in T^*\Sigma\setminus 0.$ Similarly as in the calculus for first order
edges the main point is the Fredholm property rather than the bijectivity. Recall that then there are two ways to pass to isomorphisms, provided that the corresponding analogue of the Atiyah-Bott condition is satisfied. We can fill up the Fredholm family by finite rank operators (of trace and potential type) to a $2\times 2$ block matrix family of isomorphisms, or we can add a smoothing Mellin plus Green operator family to get isomorphisms of the type of an upper left corner, cf., analogously, Theorem \ref{bje}. The second way is more convenient when we intend to explain the structure of operators in principle. This is our viewpoint here; so we have the second way in mind.
In this article we do not exhaust the consequences anyway; we are interested in a construction that produces the Fredholm property of $\eqref{2.edsy}.$ \\The operators $\eqref{2.edamp}$ belong the the corner calculus on the (open stretched) corner $\R_+\times B$ for every $(v,\zeta)\in \Sigma\times \R^d.$ The same is true of $\eqref{2.edprinc}$ for every $(v,\zeta)\in \Sigma\times \R^d\setminus \{0\}.$
Thus there is a subordinate ($\zeta$-independent) principal conormal symbol $\sigma _1(\sigma _1(a)(v,\zeta))=: f(v,z)\in \ci(\Sigma,M^\mu _{\mathcal{O}}(B,\g_{\gamma,\mu})),$ cf. Definition \ref{2.mehol} (i), with the complex covariable $z.$ Similarly as Theorem \ref{1.prep} the following result will be formulated, for simplicity, for the flat part of the edge calculus of second singularity order. The amplitude functions admit the definition of the symbols $\sigma _0$ and 
$\sigma_1, $ including their reduced variants. To be more precise, $\sigma _0(a)$ is a function on $T^*(\R_+\times s_0(B)\times \Sigma)\setminus 0$ which is of the form $\sigma _0(a)(t,b,v,\tau ,\beta ,\zeta )=t^{-\mu }\tilde{\sigma }_0(a)(t,b,v,t\tau ,\beta ,t\zeta )$ for some function $\tilde{\sigma }_0(a)(t,b,v,\tilde{\tau },\beta , \tilde{\zeta  })$ which is smooth up to $t=0.$ Here $(b,\beta )$ denote points of the cotangent bundle of $s_0(B).$ The latter dissolve locally close to $s_1(B)$ into $(r,x,y,\rho, \xi, \eta), $ and in those variables we have $\tilde{\sigma }_0(a)(t,r,x,y,v,\tilde{\tau}, \rho, \xi, \eta, \tilde{\zeta}) =r^{-\mu }\tilde{\tilde{\sigma }}_0(a)(t,r,x,y,v,r\tilde{\tau},r\rho,\xi,r\eta,r \tilde{\zeta})$ for a function $\tilde{\tilde{\sigma }}_0(a)(t,r,x,y,v,\tilde{\tilde{\tau}}, \tilde{\rho}, \xi, \tilde{\eta}, \tilde{\tilde{\zeta}})$ smooth up to $t=0,r=0.$ 
Moreover, close to $s_1(\R_+\times B \times \Sigma)=\R_+\times s_1(B)\times \Sigma$ in local splittings of variables $(t,r,x,y,v)$ we have the edge symbol  $\sigma_1(a)(t,y,v,\tau, \eta, \zeta )=t^{-\mu }\tilde{\sigma}_1(a)(t,y,v,t\tau, \eta, t\zeta):\mathcal{K}^{s ,\gamma }(X^\wedge)\rightarrow \mathcal{K}^{s-\mu ,\gamma -\mu }(X^\wedge) $ for the reduced symbol $\tilde{\sigma}_1(a)(t,y,v,\tilde{\tau}, \eta, \tilde{\zeta})$ which is smooth up to $t=0.$ In the ellipticity of operators over $\R_+\times B \times \Sigma,$ concerning the full corner calculus to construct parametrices (which is not the focus of the present paper) the assumption is $\sigma _j$-ellipticity which means non-vanishing of $\sigma_0(a)$ on $T^*(\R_+\times s_0(B) \times \Sigma)\setminus 0,$ and of $\tilde{\sigma}_0(a)$ up to $r=0, t=0,$ and $\tilde{\tilde{\sigma}}_0(a) $ up to $t=0,$ and the bijectivity of $\sigma_1(a) $ on $T^*(\R_+\times s_1(B) \times \Sigma)\setminus 0,$ and of $\tilde{\sigma}_1(a) $ up to $t=0.$
Here we need those symbols close to $t=0$ for the Fredholm property of $\eqref{2.edsy};$ so we may ignore  the properties for $t>\varepsilon $ for some $\varepsilon >0.$ In that sense we understand the conditions of the following theorem.
\begin{Thm} \label{2.prep}
Let $a(v,\zeta)$ be $\sigma _j$-elliptic close to $t=0$ for $j=0,1.$ Then for every $v\in \Sigma$ there is a discrete set $D(v)\subset \C$ intersecting $\{z\in \C:c\leq \textup{Re}\,z\leq c'\}$ in a finite set for every $c\leq c'$ such that
\begin{equation} \label{2.bijsob}
f(v,z):H^{s,\gamma }(B)\rightarrow H^{s-\mu ,\gamma -\mu }(B)
\end{equation}
are isomorphisms for all $z\in \C\setminus D(v)$ and all $s\in \R.$
\end{Thm}
\begin{proof}
We may assume that $f$ is independent of $v.$ The assumptions on the symbols imply the ellipticity of $f(z)$ in $L^\mu(B,{\bf{g}}_{\gamma ,\mu }) $ for every $z\in \C$ and hence the Fredholm property of $\eqref{2.bijsob}.$ In addition $f|_{\Gamma _\beta }$ is elliptic in the parameter-dependent class $L^\mu(B,{\bf{g}}_{\gamma ,\mu };\Gamma _\beta )$ for every $\beta \in \R,$ uniformly in compact $\beta $-intervals. Thus for every $c\leq c'$ there is a $H>0$ such that the operators $\eqref{2.bijsob}$ define isomorphisms for all $\{z\in \C:c\leq \textup{Re}\,z \leq c',|\textup{Im}\,z|\geq H\}.$ Finally, the Fredholm family $\eqref{2.bijsob}$ is holomorphic in $z\in \C.$ Thus, according to a well-known theorem on holomorphic Fredholm functions that are isomorphic for at least one point of the respective open set in the complex plane, the operators $\eqref{2.bijsob}$ are isomorphisms off a discrete set $D.$
\end{proof}
The following theorem belongs to the reasons why 
Theorem \ref{2ell} below is of interest in the corner calculus of second singularity order. However, it is not the main point of the present investigation. Therefore we only sketch the arguments of the proof; the full details are (unfortunately) voluminous. They will be presented in another paper.
\begin{Thm} \label{2.prop}
Let $a(v,\zeta)$ be $\sigma _j$-elliptic close to $t=0$ for $j=0,1.$ Then for every $v\in \Sigma$ the operators $\eqref{2.edsy}$ are Fredholm for those $\theta \in\R$ where $\eqref{2.bijsob}$ are isomorphisms for all $z\in \Gamma _{(\textup{dim}\,B+1)/2-\theta}; $ this holds for all $s\in \R.$
\end{Thm}
\begin{proof}
The assertion is an analogue of a corresponding result in the edge calculus of first singularity order. The details of the proof depend on a number of functional analytic properties of the spaces $\mathcal{H}^{s,\gamma,\theta}(B^\wedge )$ and $H^{s,\gamma}_{\textup{cone}}(B^\wedge) $ that constitute the cone spaces $\mathcal{K}^{s,\gamma,\theta}(B^\wedge),$ cf. the expression $\eqref{21.kegel},$ combined with the specific degenerate nature of the order reducing families that are involved in the definition. Those imply continuous embeddings
\begin{equation} \label{2.emb}
\langle t\rangle^{-g'}\mathcal{K}^{s',\gamma',\theta'}(B^\wedge)\hookrightarrow \langle t\rangle^{-g}\mathcal{K}^{s,\gamma,\theta}(B^\wedge)
\end{equation}
for $s'\geq s,\gamma'\geq  \gamma ,\theta' \geq \theta ,g'\geq g$ which are compact for $s'>s,\gamma' > \gamma ,\theta' >\theta ,g'>g.$ For simplicity we consider again the $v$-independent case. The conditions concerning $\sigma _0(a),\sigma _1(a)$ allow us to construct a parametrix $a^{(-1)}(\zeta )$ of analogous structure as $a(\zeta )$ in the sense that $\sigma _0(a^{(-1)})=\sigma _0^{-1}(a),\sigma _1(a^{(-1)})=\sigma _1^{-1}(a).$ Then $\sigma _2(a^{(-1)})(\zeta )$ is a family of continuous operators
\begin{equation} \label{2.invv}
\sigma _2(a^{(-1)})(\zeta ): \mathcal{K}^{s-\mu,\gamma-\mu,\theta-\mu}(B^\wedge)\rightarrow \mathcal{K}^{s,\gamma,\theta}(B^\wedge)
\end{equation}
such that
\begin{equation} \label{2.th1}
\sigma _2(a^{(-1)})(\zeta )\sigma _2(a)(\zeta )-1:\langle t\rangle^{-g}\mathcal{K}^{s,\gamma,\theta}(B^\wedge)\rightarrow \langle t\rangle^{-g'}\mathcal{K}^{s',\gamma',\theta}(B^\wedge)
\end{equation}
and
\begin{equation} \label{2.th2}
\sigma _2(a)(\zeta )\sigma _2(a^{(-1)})(\zeta )-1: \langle t\rangle^{-g}\mathcal{K}^{s-\mu,\gamma-\mu,\theta-\mu}(B^\wedge)\rightarrow \langle t\rangle^{-g'}\mathcal{K}^{s'-\mu,\gamma'-\mu,\theta-\mu}(B^\wedge) 
\end{equation}
for every $s,g$ and the given $\gamma, \theta $ for some $s'>s,g'>g$ and $\gamma'> \gamma .$ Now if the conormal symbol $f(z)$ consists of isomorphisms for all $z\in \Gamma _{(\textup{dim}\,B+1)/2-\theta},$ then $f^{-1}(z)$ can be integrated in the process of constructing $a^{(-1)}(\zeta ).$ Then, similarly as in the edge calculus of first singularity order we obtain $a^{(-1)}(\zeta )$ in such a way that $\theta $ on the right of $\eqref{2.th1}$ and $\eqref{2.th2}$ may be replaced by $\theta '.$ Applying the compact embeddings $\eqref{2.emb}$ we then obtain the claimed Fredholm property.
\end{proof}
\subsection{Ellipticity of corner conormal symbols for a prescribed weight}
Let us now consider an amplitude function $\eqref{2.edamp}$ from the edge calculus (of second singularity order) with the associated principal edge symbol $\eqref{2.edprinc}.$ Assuming that   $\eqref{2.edamp}$ is $(\sigma_0, \sigma_1)$-elliptic we construct a smoothing Mellin amplitude function $m_\theta(v,\zeta)$ such that $a(v,\zeta)+m_\theta(v,\zeta)$ is $(\sigma_0, \sigma_1,\sigma_2)$-elliptic.
\begin{Thm} \label{2ell}
Let $h\in M^\mu_\mathcal{O}(B,{\bf{g}}_{\gamma,\mu})$ such that $h|_{\Gamma_\beta}\in L^\mu(B,{\bf{g}}_{\gamma,\mu};\Gamma_\beta)$ is parameter-dependent elliptic for some real $\beta.$ Then for every fixed $\theta\in\R$ there exists an $f\in M^{-\infty}_\mathcal{O}(B,{\bf{g}}_{\gamma,\mu})$ such that 
\begin{equation} \label{2iso}
(h-f)(z): H^{s,\gamma}(B)\rightarrow H^{s-\mu,\gamma-\mu}(B) 
\end{equation}
is a family of isomorphisms for all $w\in\Gamma_{(\textup{dim}\,B+1)/2-\theta}$ and every $s\in\R.$
\end{Thm}
\begin{proof}
The above-mentioned discrete set $D\in\C$ intersects the weight line $\Gamma_{(n+1)/2-\theta}$ in at most finitely many points $\{p_1,\ldots,p_N\}$ such that $\eqref{2iso}$
is invertible for all $w\in \Gamma_{(n+1)/2-\theta}\setminus\{p_1,\ldots,p_N\}.$ Without loss of generality we assume $\theta=(\textup{dim}\,B+1)/2$ since a translation parallel to the real axis allows us to change $\theta.$ Moreover, let $N=1, p:=p_1;$ the straightforward extension of the arguments to the case of arbitrary $N$ is left to the reader. Then $\eqref{2iso}$ is invertible for all $w\in\Gamma_0\setminus\{p\}$ for some $p\in\Gamma_0.$ Since $\eqref{2iso}$ is a parameter-dependent elliptic family of operators in the edge algebra over $B,$ Fredholm and of index $0,$ there are finite-dimensional subspaces $V\subset H^{\infty,\gamma}(B),\,W\subset H^{\infty,\gamma-\mu}(B), \textup{dim}\,V=\textup{dim}\,W=:d,$ such that $V=\textup{ker}\,h(p)\subset H^{\infty,\gamma}(B),$ and $W+\textup{im}\,h(p)=H^{s-\mu,\gamma-\mu}(B)$ for every $s\in\R.$ Then for any isomorphism $k:\C^d\rightarrow W$ the row matrix

$$\begin{pmatrix}
 h & k   \\
					       
\end{pmatrix} : \Hsum{H^{s,\gamma}(B)}  
                              {\C^d}         \to
			 H^{s-\mu,\gamma-\mu}(B)$$
is surjective for all $s,$ and we have $\textup{ker}\,(h\quad k)=\textup{ker}\,h \oplus\{0\}=V \oplus\{0\}.$
Let $t_0:V\rightarrow\C^d$ be an isomorphism, and define a continuous operator $t:H^{0,\gamma}(B)\rightarrow\C^d$ by composing $t_0$ with the orthogonal projection $H^{0,\gamma}(B)\rightarrow V.$ Then $t$ extends (or restricts) to a continuous mapping $t:H^{s,\gamma}(B)\rightarrow\C^d,$ and
$$ \begin{pmatrix} h & k \cr
            t & 0 \cr \end{pmatrix}:\Hsum{H^{s,\gamma}(B)}  
                              {\C^d}         \to
			\Hsum{H^{s-\mu,\gamma-\mu}(B)}    
                             {\C^d}$$
is an isomorphism. By using the fact that that linear isomorphisms form an open set in the space of linear continuous operators, for fixed $s=s_1$ there is a $c_1>0$ such that 

\begin{equation} \label{2mat} \begin{pmatrix} h & k \cr
                t & c \cr \end{pmatrix}:\Hsum{H^{s,\gamma}(B)}  
                              {\C^d}         \to
			\Hsum{H^{s-\mu,\gamma-\mu}(B)}    
                             {\C^d}
\end{equation} 
is an isomorphism for $s=s_1,\,c:=c_1\textup{id}_{\C^d}.$ This implies that ($\ref{2mat}$) is an isomorphism for all $s\in\R$ since such a block matrix operator is an isomorphism if and only if the first row is surjective, and the second row maps the kernel of the first one isomorphically to $\C^d.$ However, $\textup{ker}\,(h\quad k)=\textup{ker}\,h=V$ is independent of $s$; therefore, the criterion is fullfilled for all $s.$ From ($\ref{2mat}$) we now produce an invertible $2\times 2$ matrix

$$\begin{pmatrix} h-kc^{-1}t & 0 \cr
                0 & c \cr \end{pmatrix}=\begin{pmatrix} 1 & -kc^{-1} \cr
                0 & 1 \cr \end{pmatrix} \begin{pmatrix}  h & k \cr
                t & c \cr \end{pmatrix} \begin{pmatrix}  1 & 0 \cr
                -c^{-1}t & 1 \cr \end{pmatrix},$$
with $1$ denoting the identity operator in $\C^d.$ Since all factors are invertible, also $$h-kc^{-1}t:H^{s,\gamma}(B)\rightarrow H^{s-\mu,\gamma-\mu}(B)
$$ is invertible. Observe that when we replace the operator $k$ in ($\ref{2mat}$) by $\delta k$ for any $\delta >0$ then the kernel of the modified first  row is isomorphically mapped by the second row to $\C^d,$ no matter how large $\delta$ is. Thus the whole construction can be repeated with $\delta k$ instead of $k$ but the same second row $(t\quad c).$\\
Let us now choose a function $\delta(w)\in\ci_0(I)$ for $I=\{w\in\Gamma_0:|w-p|<b\}$ for some $b>0,$ where $\delta(p) \neq 0.$ Then, setting

$$h_\delta(w):=\begin{pmatrix}
 h(w) & \delta(w)k   \\
					       
\end{pmatrix} : \Hsum{H^{s,\gamma}(B)}  
                              {\C^d}         \to
			 H^{s-\mu,\gamma-\mu}(B),$$   
for every $\varepsilon >0$ there are $b>0$ and $\delta\in\ci_0(I), \delta(p) \neq 0,$ with $\|h_\delta(w)-h_\delta(p)\|_{\mathcal{L}(H^{s,\gamma}(B)\oplus \C^d,H^{s-\mu,\gamma-\mu}(B))}<\varepsilon$ for all $w\in I.$\\ In fact, for sufficiently small $b$ we have $\|h(p)-h(w)\|_{\mathcal{L}(H^{s,\gamma}(B),H^{s-\mu,\gamma-\mu}(B))}<\varepsilon/2$ 
for all $w\in I,$ since $h(w)$ is continuous with values in $\mathcal{L}(H^{s,\gamma}(B),H^{s-\mu,\gamma-\mu}(B)).$ 
Moreover, 
$$\textup{sup}_{w\in I}\|\delta(p)k-\delta(w)k\|_{\mathcal{L}(\C^d,H^{s-\mu,\gamma-\mu}(B))} \leq \textup{sup}_{w\in I}|\delta(p)-\delta(w)|\|k\|_{\mathcal{L}(\C^d,H^{s,\gamma}(B))}<\varepsilon /2$$ when we choose first $\delta(w)\in\ci_0(I)$ arbitrary, $\delta (p)>0,$ and then multiply $\delta$ by a sufficiently small constant $>0$ and denote the new $\delta$ again by $\delta.$ Thus
$$ \begin{pmatrix} h(w) & \delta(w)k \cr
            t & c \cr \end{pmatrix}:\Hsum{H^{s,\gamma}(B)}  
                              {\C^d}         \to
			\Hsum{H^{s-\mu,\gamma-\mu}(B)}    
                             {\C^d}$$
is a family of isomorphisms for all $w\in I$ for sufficiently small $\varepsilon >0.$ Analogously as before we see that 
\begin{equation} \label{23}
h(w)-\delta(w)kc^{-1}t:H^{s,\gamma}(B)\rightarrow H^{s-\mu,\gamma-\mu}(B)
\end{equation}
is a family of isomorphisms for all $w\in \Gamma_0,$ first for $w\in I,$ but then, since $h(w)$ consists of isomorphisms for $w\neq p$ and $\delta\in\ci_0(I),$ also for $w\not\in I.$\\
In this consideration we have assumed that $s\in\R$ is fixed. However, the left hand side of $\eqref{23}$ consists of a family of elliptic pseudo-differential operators on $X;$ therefore, kernel and cokernel are independent of $s,$ and hence we have isomorphisms $\eqref{23}$ for all $w\in\Gamma_0,\,s\in\R.$\\
Let us interpret $f_1(w):=\delta(w)kc^{-1}t$ as an operator-valued Mellin symbol in the covariable $w\in\Gamma_0,$ with compact support in $w$ and values in operators $\in L^{-\infty}(B,{\bf{g}}_{\gamma,\mu})$ of finite rank.\\ In a final step of the proof we modify $f_1(w)$ to obtain an element $f(w)\in M^{-\infty}_{\mathcal{O}}(B,{\bf{g}}_{\gamma,\mu})$ that approximates $f_1(w)$ in such a way that $\eqref{2iso}$ are isomorphisms for all $w\in\Gamma_0.$ First Theorem \ref{2.conv} gives us $f_{(\varepsilon)}(w):=V(\psi_{\varepsilon})f_1(w)\in M^{-\infty}_{\mathcal{O}}(B,{\bf{g}}_{\gamma,\mu}),$ and $f_{(\varepsilon)}\rightarrow f_1$ as $\varepsilon\rightarrow 0$ in the topology of $L^{-\infty}(B,{\bf{g}}_{\gamma,\mu};\Gamma_0).$ We will show that we may set $f(w)=f_{(\varepsilon)}(w)$ for any fixed sufficiently small $\varepsilon >0.$ It is evident that for any fixed compact interval $K\subset\Gamma_0$ containing the point $p$ there is an $\varepsilon(K)>0$ such that for all $0<\varepsilon<\varepsilon(K)$ we have isomorphisms $\eqref{2iso}$ for all $w\in K.$ To argue for $w\in \Gamma_0 \setminus K$ we employ Theorem \ref{2.meinv}, i.e., there is an element $h^{-1}(w) \in M^{-\mu}_S(B,{\bf{g}}_{\gamma,\mu})$ for some discrete Mellin asymptotic type $S$ with $\pi_\C S\cap\Gamma_0=\{p\}$ such that $h^{-1}(w)h(w)=1.$ This gives us  family of continuous operators
\begin{equation} \label{2map}
1-h^{-1}(w)f_{(\varepsilon)}(w):H^{s,\gamma}(B)\rightarrow H^{s,\gamma}(B)
\end{equation}
parametrised by $w\in\Gamma_0\setminus K.$ Without loss of generality we take $K$ so large that $\textup{supp}\,f_1\subseteq K_1$ for some subinterval $K_1\subset \textup{int}\,K.$ Let $\chi (w)\in \ci(\Gamma_0)$ be any function which is equal to $0$ in a neighbourhood of $\textup{supp}\,f_1$ and $1$ outside an open $U\subset\Gamma_0,\,K_1\subset U$ with $\overline U\subset K.$ Then $f_{(\varepsilon)}\rightarrow f_1$ gives us $\chi f_{(\varepsilon)}\rightarrow 0$ as $\varepsilon\rightarrow 0$ in the space $\mathcal{S}(\Gamma_0,L^{-\infty}(B,{\bf{g}}_{\gamma,\mu})).$ Using that $h^{-1}(w)$ is a parameter-dependent family of operators in $L^{-\mu}(B,{\bf{g}}_{\gamma-\mu,-\mu})$ with parameter $w\in\Gamma_0 \setminus\{p\}$ we also obtain $h^{-1}\chi f_{(\varepsilon)}\rightarrow 0$ in $\mathcal{S}(\Gamma_0,L^{-\infty}(B,{\bf{g}}_{\gamma,0})).$
It follows that there is an $\tilde{\varepsilon}(K)>0$ such that $\eqref{2map}$ are isomorphisms for all $0<\varepsilon <\tilde{\varepsilon}(K)$ and all $w\in\Gamma_0 \setminus K.$
Thus we obtain altogether isomorphisms $\eqref{2iso}$ for $f(w)=f_{(\varepsilon)}(w),\,0<\varepsilon <\textup{min}\,(\varepsilon(K),\tilde{\varepsilon}(K))$ for all $w\in\Gamma_0.$
\end{proof}
Note that an analogue of Theorem \ref{2.conv} in the simpler case of a smooth compact manifold rather than $B$ can also be deduced from the the technique of Witt \cite{Witt3}, see also the paper \cite{Mali1}.


\end{document}